\documentclass{amsart}
\usepackage{amsmath}
\usepackage{amssymb}
\usepackage{amsthm}
\usepackage{amsaddr}
\usepackage{mathrsfs}
\usepackage[all,cmtip]{xy}
\usepackage{times}
\usepackage{graphicx}
\usepackage{enumerate}
\numberwithin{equation}{section}

\usepackage{graphicx}
\usepackage{tikz-cd}
\usepackage[all,cmtip]{xy}
\usepackage{mathrsfs}
\usepackage{mathtools}

\def\lhd{\triangleleft}
\def\ra{\rightarrow}

\def\iff{\Longleftrightarrow}

\def\emb{\hookrightarrow}

\def\A{{\mathcal A}}
\def\B{{\mathcal B}}

\def\C{{\mathcal C}}

\def\dom{{\mathbf d}}

\def\frakL{{\mathfrak L}}
\def\L{{\mathcal L}}

\def\P{{\mathcal P}}

\def\relbeta{\, \beta \,}

\def\lam{\lambda}

\def\colim{\operatorname{colim}}

\def\id{\operatorname{id}}

\def\homogpd{\textbf{OGpd}}

\newcommand{\Mod}[1]{\operatorname{Mod}_{#1}}

\def\dom{\mathbf{d}}
\def\ran{\mathbf{r}}
\newcommand{\down}[1]{(#1)^{\downarrow}}

\def\leq{\leqslant}
\def\geq{\geqslant}
\def\<{\langle}
\def\>{\rangle}
\def\ol{\overline}

\newtheorem{theorem}{Theorem}[section]
\newtheorem{prop}[theorem]{Proposition}
\newtheorem{lemma}[theorem]{Lemma}
\newtheorem{cor}[theorem]{Corollary}

\theoremstyle{definition}

\newtheorem{definition}[theorem]{Definition}

\title{The homology of principally directed ordered groupoids}

\author{B. O. Bainson and N. D. Gilbert}

\address{
School of Mathematical and Computer Sciences\\
and the Maxwell Institute for the Mathematical Sciences,\\
Heriot-Watt University, Edinburgh EH14 4AS, U.K.}
\email{bob30@hw.ac.uk, N.D.Gilbert@hw.ac.uk}

\date{}
\thanks{Some of these results are presented in a different form as part of the first author's PhD thesis \cite{bob}.  
The support of a MACS Global Platform Studentship from Heriot-Watt University is gratefully acknowledged.}


\begin{document}

\begin{abstract} 
We present some homological properties of a relation $\beta$ on ordered groupoids that generalises the minimum group
congruence for inverse semigroups.   When $\beta$ in a transitive relation on an ordered groupoid $G$, the
quotient $G / \beta$ is again an ordered groupoid, and construct a pair of adjoint functors between the module categories of
$G$ and of $G / \beta$.  As a consequence, we show that the homology of $G$ is completely determined by that of $G / \beta$,
generalising a result of Loganathan for inverse semigroups.
\end{abstract}


\maketitle

\thispagestyle{empty}

\section{Introduction}
This paper studies some homological properties of a quotient construction for ordered groupoids determined by a certain
relation $\beta$ that generalises the minimal group congruence $\sigma$ on an inverse semigroup.  
Modules for inverse semigroups, and the cohomology of an inverse semigroup, were first defined by Lausch in \cite{Lau},
and the cohomology used to classify extensions.   An approach based on the cohomology of categories was
then given by Loganathan \cite{Log}, who showed that Lausch's cohomology of an inverse semigroup $S$ was equal to
the cohomology of a left-cancellative category $\frakL(S)$ naturally associated to $S$.  Loganathan proves a number of results
relating the cohomology of $S$ with that of its semilattice of idempotents $E(S)$ and of its maximum group image
$S/ \sigma$.  He also considers the homology of $S$, but the treatment is brief since \cite[Proposition 3.5]{Log} shows that
the homology of $S$ is completely determined by the homology of the group $S / \sigma$.

Ordered groupoids and inverse semigroups are closely related, since any inverse semigroup can be considered as
a particular kind of ordered groupoid -- an \textit{inductive} groupoid -- and this correspondence in fact gives rise to an isomorphism
between the category of inverse semigroups and the category of inductive groupoids.  This is the Ehresmann-Schein-Nambooripad 
Theorem (see \cite[Theorem 4.1.8]{LwBook}).  This close relationship has been exploited in the use of ordered
groupoid techniques to prove results about inverse semigroups (see \cite{HMM, LwBook, LMP, St}) and has been the motivation behind 
various generalisations of results about inverse semigroups to the wider class of ordered groupoids (see \cite{AGM,Gi3,Lw1}).

In this paper we revisit Loganathan's results on the homology of inverse semigroups, and we are led to consider the relation 
$\beta$ on an ordered groupoid $G$ defined as follows: two elements of $G$ are $\beta$--related if and only if they have a
lower bound in $G$.  This relation is trivially reflexive and symmetric but need not be transitive: when it is, we say that
$G$ is a principally directed ordered groupoid, a choice of terminology justified in Lemma \ref{princ_dir} below. 
The $\beta$--relation and the class of principally directed ordered groupoids
featured in \cite{Gi3} (but there called  $\beta$--transitive ordered groupoids), in the study of the structure of inverse semigroups $S$ with zero.  In this setting, $S^* = S \setminus \{ 0 \}$
can be considered as an ordered groupoid, and $S^*$ is then principally directed if and only if $S$ is \textit{categorical at zero}: that
is, whenever $a,b,c \in S$ and $abc=0$ then either $ab=0$ or $bc=0$.  The structure theorem of Gomes and Howie 
\cite{GomHo} for
strongly categorical inverse semigroups with zero can then be deduced from a more general result on principally directed
ordered groupoids \cite[section 4.1]{Gi3}.  In this paper, the significance of the transitivity of $\beta$ is that it permits the construction
of a pair of adjoint functors between the module categories of $G$ and of $G / \beta$.  The left adjoint is simply the colimit
over $E(G)$.  The right adjoint expands a $G / \beta$--module to a $G$--module.  These constructions are discussed in  section
\ref{inf_and_colim}, and generalise the key ingredients of Loganathan's treatment of the homology of inverse semigroups
in \cite{Log}.  The fact that the homology of a principally directed ordered groupoid $G$ is determined by the homology of the
quotient $G / \beta$ then follows readily in section \ref{homology_of_og}.

\section{Ordered Groupoids}
\label{ord_gpd}
A groupoid $G$ is a small category in which every morphism is invertible. The set of identities of $G$ is denoted $E(G)$, following the customary notation for the set of idempotents in an  inverse semigroup. We write $g \in G(e,f)$ when $g$ is a morphism starting at $e$ and ending at $f$.  We regard a groupoid as an algebraic structure comprising its morphisms, and compositions of morphisms as a partially defined binary operation (see \cite{HiBook}, \cite{LwBook}). The identities are then written as $e=g\dom =gg^{-1}$ and $f=g\ran=g^{-1}g$ respectively. A groupoid map $\theta:G\ra H$ is just a functor.

\begin{definition}
An ordered groupoid is a pair $(G, \leq)$ where $G$ is a groupoid and $\leq$ is a partial order defined on $G$, satisfying the following axioms:
\begin{enumerate}
\item[OG1]  $x \leq y  \Rightarrow x^{-1} \leq y^{-1}$, for all $x, y \in G$.
\item[OG2] Let $x, y, u, v\in G$ such that $x\leq y$ and $u\leq v$. Then $xu\leq yv$ whenever the compositions $xu$ and $yv$ exist.
\item[OG3] Suppose $x\in G$ and $e\in E(G)$ such that $e\leq x\mathbf{d}$, then there is a unique element $(e|x)$ called the \textit{restriction} of $x$ to $e$ such that $(e|x)\mathbf{d}= e$ and $(e|x)\leq x$.
\item[OG4] If $x\in G$ and $e \in E(G)$ such that $e \leq x \mathbf{r}$, then there exist a unique element $(x|e)$ called the \textit{corestriction} of $x$ to $e$ such that $(x|e) \mathbf{r}= e$ and $(x|e) \leq x$. 
\end{enumerate} 
It is easy to see that OG3 and OG4 are equivalent: if OG3 holds then we may define a corestriction $(x|e)$ by $(x|e) = (e|x^{-1})^{-1}$.
\end{definition}
An ordered functor $\phi: G \ra H$ of ordered groupoids is an order preserving groupoid--map, that is $g\phi \leq h\phi$ if $g\leq h$. Ordered groupoids together with ordered functors constitute the category of ordered groupoids, $\homogpd{}$.  

Suppose $g,h \in G$ and that the greatest lower bound $\ell$ of $g\textbf{r}$ and $h\textbf{d}$ exist, then we define the {\it pseudoproduct} of $g$ and $h$ by $g\ast h= (g|\ell)(\ell|h)$.  An ordered groupoid is called \textit{inductive} if the pair $(E(G),\leq)$ is a meet semilattice. In an inductive groupoid $G$, the pseudoproduct is everywhere defined and $(G,\ast)$ is then an inverse semigroup:
see \cite[Theorem 4.1.8]{LwBook}

To any ordered groupoid $G$ we associate a category$\frakL(G)$ as follows. The objects of $\frakL(G)$ are the identities of $G$ and morphisms are given by pairs $(e,g) \in E(G) \times G$ where $g\dom \leq e$, with $(e,g)\dom=e$ and $(e,g)\ran=g\ran $. The composition of morphisms is defined by the partial product $(e,g)(f,h)= (e, g\ast h)=(e,(g|h \dom)h)$ whenever $g\ran=f$. It is easy to see that $\frakL(G)$ is left cancellative.  
This construction originates in the work of Loganathan \cite{Log}, and forms the basis of the treatment in \cite{Log} of the cohomology of inverse semigroups.

\section{Principally directed ordered groupoids}
Let $G$ be an ordered groupoid. The relation $\beta$ on $G$ is defined by
\[ g \relbeta h \; \iff \; \text{there exists} \; k \in G \; \text{with} \; k \leq g \; \text{and} \; k \leq h \,.\]
$\beta$ is evidently reflexive and symmetric but need not be transitive: we shall be concnerned with the class of ordered groupoids 
for which  $\beta$ is
indeed transitive, and thus an equivalence relation.  We shall denote the  $\beta$--class of $g \in G$ by
$g \beta$.  A \textit{principal order ideal} is a subset of $G$ of the form $\{ g \in G : g \leq t \}$ for some $t \in G$, and will
be denoted by $\down{t}$.

\begin{lemma}{\cite[section 2.2]{Gi3}}
\label{princ_dir}
The $\beta$--relation on an ordered groupoid $G$ is transitive if and only if every principal order ideal in $G$ is a directed set.
\end{lemma}

\begin{proof}
Suppose that $\beta$ is transitive, and that $g,h \in \down{t}$.  Then $g \relbeta t \relbeta h$ and so $g \relbeta h$,
and there exists $k \in G$ with $k \leq g$ and $k \leq h$: hence $k \in \down{t}$ and $\down{t}$ is a directed set.
Conversely, suppose that $g \relbeta t \relbeta h$: then there exist $k,l \in G$ with $k \leq g$, $k \leq t$, $l \leq t$ and $l \leq h$.
In particular, $k,l \in \down{t}$, and if $\down{t}$ is a directed set then there exists $c \leq t$ with $c \leq k$ and
$c \leq l$.  Then $c \leq g$ and $c \leq h$, and so $g \relbeta h$.
\end{proof}

\begin{definition}
An ordered groupoid in which every principal order ideal is a directed set will be called \textit{principally
directed}.  This terminology is consistent with that of \cite{Lw1}.
\end{definition}

It is clear that if $G$ is principally directed then so is its poset of identities $E(G)$.   However, the converse is false.  
Let $A$ and $B$ be groups with a common subgroup $C$ and let $i : C \emb A$ and $j: C \emb B$ be the inclusions. Consider the
semilattice $\{0,e,f,1\}$ with $e,f$ incomparable, and define a semilattice of groups $G$ by $G_1=C, G_e = A, G_f=B$ and
$E(G)= A \times B$ and with the obvious structure maps.  Then $ci \relbeta c \relbeta cj$ for all $c \in C$, but $ci$ and $cj$ are not
$\beta$--related.

\begin{prop}{\cite[Proposition 2.2]{Gi3}}
\label{quot_gpd}
If $G$ is a principally directed ordered groupoid then the quotient set $G/ \beta$ is a groupoid.
\end{prop}

The groupoid structure on $G / \beta$ is inherited from $G$ in the following way.  If $g,h \in G$ and $g^{-1}g \relbeta hh^{-1}$
then there exists $f \in E(G)$ with $f \leq g^{-1}g$ and $f \leq hh^{-1}$, and the composition of the $\beta$--classes
of $g$ and $h$ is then defined by
\[ (g \beta)(h \beta)= [(g|f)(f|h)] \beta \,.\]
This is easily seen to be independent of any choices made for $f$ and for representatives of $g \beta$ and $h \beta$: see
\cite[section 2.2]{Gi3} for further details.  However, there is no natural ordering inherited by $G / \beta$, and so we regard
$G / \beta$ as trivially ordered.  Lawson \cite[Theorem 20]{Lw1} states Proposition \ref{quot_gpd} for the special case of 
\textit{principally inductive} ordered groupoids. 

\section{Expansion and colimits of modules}
\label{inf_and_colim}
Let $G$ be an ordered groupoid, and $\frakL(G)$ its associated left-cancellative category.  A $G$--module is defined to be an
$\frakL(G)$--module, that is, a functor $\A$ from $\frakL(G)$ to the category of abelian groups.  A $G$--module $\A$ is thus comprised
of a family of abelian groups $\{ A_e : e \in E(G) \}$ together with a group homomorphism $\alpha_{(e,g)}: A_e \ra A_{g^{-1}g}$
for each arrow $(e,g)$ of $\frakL(G)$.  We shall often denote $a \alpha_{(e,g)}$ by $a \lhd (e,g)$.  Morphisms of $G$--modules (called $G$--\textit{maps}) are natural transformations of functors, and so we obtain a category $\Mod{G}$ of $G$--modules and $G$--maps.

Suppose that $G$ is principally directed. No ordering is prescribed for the quotient groupoid $G / \beta$ and so $\frakL(G / \beta) = G / \beta$.  If $\B$ is a $(G / \beta)$--module
then we can {\em expand} $\B$ to obtain an $\frakL(G)$--module $\B^{\uparrow}_{\beta}$ with homomorphisms
$\mu_{(e,g)}$ as follows:
\begin{itemize}
\item for $e \in E(G)$ we have $(\B^{\uparrow}_{\beta})_e = B_{e \beta}$,
\item if $e \geq f$ then $e \beta = f \beta$ and $\mu_{(e,f)} = \id$,
\item for $x,y \in E(G)$ and for each $g \in G(x,y)$, the map $\mu_{(x,g)} : B_{x \beta} \ra B_{y \beta}$ is just the map
$\mu_{g \beta} : B_{x \beta} \ra B_{y \beta}$ determined by $\B$.
\end{itemize}
This defines the \textit{expansion functor} $\Mod{G / \beta} \ra \Mod{\frakL(G)}$ since, if $\xi : \B \ra \B'$ is a $G / \beta$--map then we have a
commutative diagram
\[ \xymatrixcolsep{3pc}
\xymatrix{
B_{e \beta} \ar[rr]^{\xi_{e \beta}} \ar@{=}[d] \ar@/_3pc/[dd]_{\lhd (e,g)} &&  B'_{e \beta} \ar@{=}[d] \\
B_{(gg^{-1})\beta}  \ar[rr]^{\xi_{e \beta}} \ar[d]^{\lhd g \beta} && B'_{(gg^{-1}\beta)}  \ar[d]^{\lhd g \beta} \\
B_{(g^{-1}g)\beta} \ar[rr]_{\xi_{(g^{-1}g)\beta}} && B'_{(g^{-1}g)\beta}
}
\]
and so we obtain an $\frakL(G)$--map $\xi^{\uparrow}_{\beta} : \B^{\uparrow}_{\beta} \ra (\B')^{\uparrow}_{\beta}$ with
$(\xi^{\uparrow}_{\beta})_e = \xi_{e \beta}$.  

\begin{lemma}
\label{infl_preserves_epis}
The expansion functor $\Mod{G / \beta} \ra \Mod{\frakL(G)}$ preserves epimorphisms.
\end{lemma}

 \begin{proof}
Epimorphisms in $\Mod{}$ are given by families of surjections, and so if $\xi$ is an epimorphism in
$\Mod{G / \beta}$ then so is $\xi^{\uparrow}_{\beta}$ in $\Mod{\frakL(G)}$.
\end{proof}

The expansion functor is implicit in \cite{Log} for the case in which $\beta$ is replaced by the minimal group congruence $\sigma$ on an inverse semigroup.  We now generalise \cite[Lemma 3.4]{Log} and show that the expansion functor $\Mod{G / \beta} \ra \Mod{\frakL(G)}$ for a principally directed ordered groupoid $G$ admits a left adjoint.

Suppose that $\A$ is an $\frakL(G)$--module.  We consider the restriction of $\A$ to an $E(G)$--module,  involving the same
abelian groups $A_e, (e \in E(G))$ but using only the maps $\alpha_{(e,f)} : A_e \ra A_f$ from $\A$.  The colimit
$\colim^{E(G)} \A$ is then a direct sum
\[ \colim^{E(G)} \A = \bigoplus_{x \in E(G) / \beta} L_x \]
indexed by the $\beta$--classes in $E(G)$, and so determines an $E(G/ \beta)$--module $\L$ with $C_{e \beta} = L_{e \beta}$
and with trivial action, since $E(G / \beta)$ is a trivially ordered poset.   We shall allow ourselves a small abuse of notation,
and denote $\L$ by $\colim^{E(G)} \A$.

\begin{prop}
\label{colim_is_module}
If $G$ is principally directed and $\A$ is a $G$--module then $\colim^{E(G)} \A$ is a $G / \beta$--module.
\end{prop}

\begin{proof}
Let $\colim^{E(G)} \A = \bigoplus L_x$ as  above, let $\alpha_e : A_e \ra L_{e \beta}$ be the canonical map.
Suppose that  $\ol{a} \in L_{e \beta}$ with $\ol{a} = a \alpha_e$ for some in $A_e$, and $g \in G$ with $gg^{-1} \relbeta e$. Then $gg^{-1}$ and $e$ have a lower bound $\ell$, and we define
an action of $g \beta$ on $\ol{a}$ by
\begin{equation} 
\label{glamaction}
\ol{a} \lhd g \beta = (a \alpha_{(e,\ell)} \lhd (\ell|g))\alpha_z 
\end{equation}
where $z = (\ell | g)\ran$.  We have to check that this definition is independent of the choices made for $\ell, a$ and $g$.

If we choose a  different lower bound $\ell'$ of $gg^{-1}$ and $e$, then $\ell$ and $\ell'$ are $\beta$--related (using the 
transitivity of $\beta$) and so have a lower bound $\ell''$.
It is sufficient to show, for independence from the choice of $\ell$,  that the outcome of \eqref{glamaction} is unchanged 
by descent in the partial order, in the following sense.

Suppose that $a \in A_e$, $gg^{-1}=e$ and that $f \leq e$.  Let $y = g^{-1}g$ and $z =(f|g)\ran$.  Then \eqref{glamaction}
gives $a \alpha_e \lhd g \beta = (a \lhd g) \alpha_y$.  If we base the calculation at $f$ we obtain 
$(a \alpha_{(e,f)} \lhd (f|g))\alpha_z$.  But in $\frakL(G)$,
\[ (e,f)(f,(f|g)) = (e,(f|g)) = (e,(e|g))(y,z) \]
and so $a \alpha_{(e,f)} \lhd (f|g) = (a \lhd g) \alpha_{(y,z)}$.
Hence
\[ (a \alpha_{(e,f)} \lhd (f|g))\alpha_z = (a \lhd g) \alpha_{(y,z)} \alpha_z = (a \lhd g) \alpha_y \,. \]
Therefore the outcome of \eqref{glamaction} is independent of the choice of $\ell$.

We now consider the choice of a preimage for $\ol{a}$.  Suppose that $a \alpha_e = b \alpha_x$.  Then $e \beta x$ and so
$e$ and $x$ have a 
lower bound $u$ with $\ol{a} = a \alpha_{(e,u)} \alpha_u = b \alpha_{(x,u)} \alpha_u$.  So again it suffices to check what happens if we 
apply \eqref{glamaction} at $u$.  We have
\[ 
\ol{a} \lhd g \beta = (a \lhd g) \alpha_y 
= (a \alpha_{(e,u)}) \lhd(u|g)) \alpha_z \]
where now $z = (u|g)\ran$.  But as before, $a \alpha_{(e,u)} \lhd (u|g) = (a \lhd g) \alpha_{(y,z)}$ and $\alpha_{(y,z)} \alpha_z = \alpha_y$.
Hence the definition in \eqref{glamaction} is independent of the choice of $a$.

Finally, suppose that $g \relbeta h$.  Then $gg^{-1} \relbeta hh^{-1}$ and so $gg^{-1}$ and $hh^{-1}$ have a lower bound
$v \in E(G)$.  Then
$g \beta = (v|g) \beta = (v|h)\beta = h \beta$, and
acting with $(v|g)$ in \eqref{glamaction} we obtain
\begin{align*}
\ol{a} \lhd (v|g)\beta &= (a \alpha^e_v \lhd (v |g)) \alpha_z \\
&= (a \lhd g) \alpha_{(y,z)} \alpha_z \\
&= (a \lhd g) \alpha_y \,.
\end{align*}
Hence the definition in \eqref{glamaction} is independent of the choice of $g$, and we have a well-defined action of
$G / \beta$ on $\colim^{E(G)} \A$.
\end{proof}

Let $\B$ be a $G/ \beta$--module, let $\A$ be a $G$--module, and suppose
 that we are given a map $\phi : \A \ra \B^{\uparrow}_{\beta}$, with components
$\phi_e : A_e \ra B_{e \beta}, (e \in E(G))$.  Whenever $e \geq f$ we have a commutative triangle
\[ \xymatrixcolsep{3pc}
\xymatrix{
A_e \ar[rr]^{\alpha_{(e,f)}} \ar[dr]_{\phi_e } &&  A_f \ar[dl]^{\phi_f}\\
& B_{e \beta} & }
\]
(in which $B_{e \beta} = B_{f \beta}$) and so the $\phi_e$ induce a family of maps $\psi$ with
$\psi_{e \beta} : L_{e \beta} \ra \B_{e \beta}$
and, if $\alpha_e : A_e \ra \colim^{E(G)} \A$ is the canonical map, then $\phi_e = \alpha_e \psi_{e \beta}$.  Therefore $\psi$ determines $\phi$, and we have the following Corollary of  Proposition \ref{colim_is_module}.

\begin{cor}
If $G$ is principally directed then $\psi : \colim^{E(G)} \A \ra \B$ is a $G / \beta$--map, and 
$\phi \mapsto \psi$ is an injection
\begin{equation}
\label{adjoint1}
\rho:  \Mod{\frakL(G)}(\A,\B^{\uparrow}_{\beta}) \ra \Mod{G / \beta}(\colim^{E(G)} \A, \B) \,. \end{equation}
\end{cor}

\begin{theorem}
\label{infl_adjoint}
Let $G$ be a principally directed ordered groupoid.  Then the functor $\colim^{E(G)} : \Mod{G} \ra \Mod{G / \beta}$
is left adjoint to the expansion functor.
\end{theorem}

\begin{proof}
We wish to construct a function
\begin{equation}
\label{adjoint2}
\tau:   \Mod{G / \beta}(\colim^{E(G)} \A , \B) \ra \Mod{G}(\A,\B^{\uparrow}_{\beta}) \,. \end{equation}
that will be inverse to $\rho$ in \eqref{adjoint1}.
For $e \in E(G)$ and $\psi: \colim^{E(G)} \A \ra \B$, consider the composition
\[ \xymatrixcolsep{3pc}
\xymatrix@1{
A_e \ar[r]_{\alpha_e} & L_{e \beta} \ar[r]_(.35){\psi_{e \beta}} & B_{e \beta} = (\B^{\uparrow}_{\beta})_e \,.} \]
This composition is a $G$--map since, for $a \in A_e$,
\begin{align*}
(a \alpha_e \psi) \mu_{g \beta} &= (a \alpha_e \lhd g \beta) \psi_{(g^{-1}g)\beta} \\
\intertext{and, evaluating the $g \beta$ action using \eqref{glamaction} with $\ell = gg^{-1}$,}
&= (a \lhd (e,g)) \alpha_{(g^{-1}g)\beta}\psi_{(g^{-1}g)\beta}
\end{align*}
and so the diagram
\[ \xymatrixcolsep{3pc}
\xymatrix{
A_e \ar[r]^{\alpha_e} \ar[d]_{\lhd (e,g)} & L_{e \beta} \ar[d]_{\lhd g \beta} \ar[r]^{\psi_{e \beta}} &  B_{e \beta}  \ar[d]^{\lhd g \beta} \\
A_{g^{-1}g} \ar[r]_{\alpha_{g^{-1}g}} & L_{(g^{-1}g) \beta} \ar[r]_{\psi_{(g^{-1}g) \beta}} & B_{(g^{-1}g) \beta} }
\]
commutes.  Now the injection $\rho$ in \eqref{adjoint1} carries $(\alpha_e \psi_{e \beta})$ to $\psi$ and so $\tau \rho$ is the identity.
A $G$--map $\phi : \A \ra \B^{\uparrow}_{\beta}$ is carried by $\rho$ to the induced map $\psi: \L \ra \B$,
where $\phi_e = \alpha_e \psi_{e \beta}$.  But $\tau$ carries $\psi$ precisely to this composition, and so
$\rho \sigma$ is aslo the identity, and so in the principally directed case, \eqref{adjoint1} and \eqref{adjoint2} exhibit a natural
bijection and its inverse.
\end{proof}

\subsection{Composition of colimits}
If $G$ is principally directed, then we have seen in Proposition \ref{colim_is_module} that, for every $G$--module $\A$,
the colimit $\L = \colim^{E(G)} \A$ can be considered as a $G / \beta$--module.  Since $G / \beta$ need not be connected,
$\colim^{G / \beta} \L$ decomposes in general into a direct sum 
$\colim^{G / \beta} \L = \bigoplus_{p \in \pi_0 (G / \beta)} C_p$ indexed by the connected components of $G / \beta$.  We can therefore form $\colim^{G / \beta} \L$, with canonical
maps $\psi_{e \beta} : L_{e \beta} \ra C_{e \beta}$, where $e \beta$ is the connected component of $e \in G$ in the quotient groupoid
$G / \beta$.

\begin{prop}
\label{comp_of_colim}
The colimit $\colim^{G / \beta} (\colim^{E(G)} \A)$ is naturally isomorphic to $\colim^{\frakL(G)} \A$.
\end{prop}

\begin{proof}
We show that $\colim^{G / \beta} \L$ has the universal property required of $\colim^{\frakL(G)} \A$.  As above, 
we have $\alpha_e : A_e \ra L_{e \beta}$ and a commutative diagram
\[ \xymatrixcolsep{2pc}
\xymatrix{
A_e \ar[rr]^{\alpha_e} \ar[dd]_{\lhd (e,g)} &&  L_{e \beta} \ar[dd]_{\lhd g \lam} \ar@/^/[rrd]^{\psi_{e \beta}} && \\
&&&& \colim^{G / \beta} \L \\
A_{g^{-1}g} \ar[rr]_{\alpha_{g^{-1}g}} && L_{(g^{-1}g)\beta} \ar@/_/@<-1ex>[rru]_{\psi_{(g^{-1}g)\beta}} \\}
\]
from which we extract the commutative triangles
\[ \xymatrixcolsep{2pc}
\xymatrix{
A_e \ar[rrd]^{\alpha_e \psi_{e \beta}} \ar[dd]_{\lhd (e,g)} &&  \\
&& \colim^{G / \beta} \L \\
A_{g^{-1}g} \ar@<-1ex>[rru]_{\phantom{xxx}\alpha_{g^{-1}g} \psi_{(g^{-1}g)\beta}} && }
\]
Suppose we are given a family of maps $\mu_e : A_e \ra M$ to some abelian group $M$ making commutative triangles
\[ \xymatrixcolsep{2pc}
\xymatrix{
A_e \ar[rrd]^{\mu_e} \ar[dd]_{\lhd (e,g)} &&  \\
&& M \\
A_{g^{-1}g} \ar@<-1ex>[rru]_{\phantom{xxx}\mu_{g^{-1}g} } && }
\]
In particular, for $f \leq e$ we have 
\[ \xymatrixcolsep{2pc}
\xymatrix{
A_e \ar[rrd]^{\mu_e} \ar[dd]_{\alpha_{(e,f)}} &&  \\
&& M \\
A_f \ar@<-1ex>[rru]_{\phantom{xxx}\mu_f } && }
\]
and hence a unique family of maps $\delta_{e \beta}: L_{e \beta} \ra M$ making the diagrams
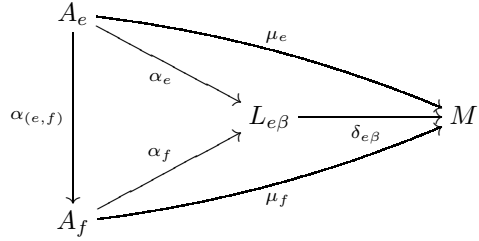
\[ \xymatrixcolsep{2pc}
\xymatrix{
A_e \ar[rrd]_{\alpha_e} \ar[dd]_{\alpha_{(e,f)}} \ar@/^/[rrrrd]^{\mu_e} &&&&  \\
&& L_{e \beta} \ar[rr]_{\delta_{e \beta}} && M \\
A_f \ar[rru]^{\alpha_f } \ar@/_/[rrrru]_{\mu_f} &&&& }
\]
commute.

Now consider the action of $g \beta$ on $\ol{a} = a \alpha_e \in L_{e \beta}$.  From \eqref{glamaction}
\begin{align*}
(\ol{a} \lhd g \beta) \delta_{(g^{-1}g)\beta} & = (a \alpha^e_{\ell} \lhd (\ell|g))\alpha_z \delta_{z \beta} \\
&= (a \alpha^e_{\ell} \lhd (\ell|g))\mu_z \\
&= a \mu_e \qquad
\text{(since $\mu_e = \alpha_{(e,(\ell,g))}\mu_{g^{-1}g}$)} \\
&= a \alpha_e \delta_{e \beta} \\ &= \ol{a} \delta_{e \beta}.
\end{align*}

Hence the triangles 
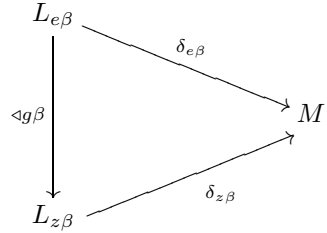
\[ \xymatrixcolsep{3pc}
\xymatrix{
L_{e \beta} \ar[rrd]^{\delta_{e \beta}} \ar[dd]_{\lhd g \beta} &&  \\
&& M \\
L_{z \beta} \ar@<-1ex>[rru]_{\phantom{xxx}\delta_{z \beta}} && }
\]
commute and induce a unique map $\delta: \colim^{G / \beta} \L \ra M$ making the diagram 
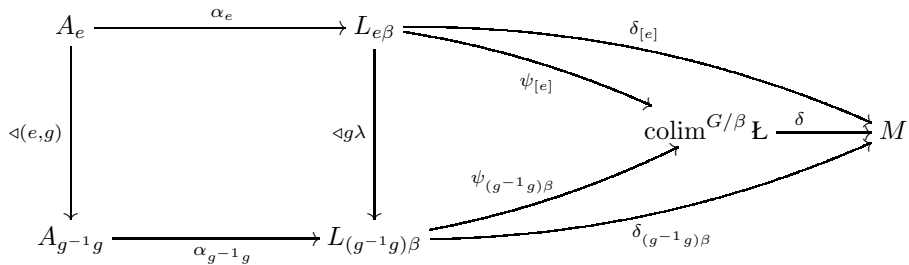
\[ \xymatrixcolsep{3pc}
\xymatrix{
A_e \ar[rr]^{\alpha_e} \ar[dd]_{\lhd (e,g)} &&  L_{e \beta} \ar[dd]_{\lhd g \lam} \ar@/^/[rrd]_{\psi_{[e]}} \ar@/^1pc/[rrrd]^{\delta_{[e]}} &&& \\
&&&& \colim^{G / \beta} \L \ar[r]^{\delta} & M\\
A_{g^{-1}g} \ar[rr]_{\alpha_{g^{-1}g}} && L_{(g^{-1}g)\beta} \ar@/_/[rru]^{\psi_{(g^{-1}g)\beta}\phantom{xx}} \ar@/_1pc/[rrru]_{\phantom{xxx} \delta_{(g^{-1}g)\beta}} \\}
\]
commute, since $L_{z \beta} = L_{(g^{-1}g)\beta}$.
\end{proof}

\section{The homology of principally directed ordered groupoids}
\label{homology_of_og}
The  functors $H_n(G,-), n \geq 0$, for a fixed ordered groupoid $G$ (or equivalently, for the left-cancellative category $\frakL(G)$), may be 
characterized as functors 
$\Mod{G} \ra  \mathbf{Ab}$ by the following properties:
\begin{enumerate}
\item[(a)] $H_n(G,-), n \geq 0$ is a homological extension of the colimit $\colim^{\frakL(G)}$, so that 
\begin{itemize}
\item $H_0(G,\A) = \colim^{\frakL(G)}(\A)$,
\item for any short exact sequence $\A \ra \B \ra \C$ of $G$--modules and for each $n \geq 0$, there exists a natural homomorphism
$d_n : H_{n+1}(G,\C) \ra H_n(G \A)$ inducing an exact sequence 
\[ \dots \ra H_{n+1}(G,\C) \ra H_n(G, \A) \ra H_n(G,\B) \ra H_n(G,\C)  \ra H_{n-1}(G, \A) \ra \dots \]
\end{itemize}
\item[(b)] $H_n(G,\P) = 0$ for all $n>0$ and all projective modules $\P$.
\end{enumerate}

\begin{theorem}
\label{homology_is_level}
For any principally directed ordered groupoid $G$ and $G$--module $\A$, and any $n \geq 0$, the homology groups $H_n(G,\A)$ and
$H_n(G / \beta,  \colim^{E(G)} \A)$ are isomorphic.
\end{theorem}

\begin{proof}
We consider the functor $\Mod{\frakL(G)} \ra \mathbf{Ab}$ given by \[ \A \mapsto H_n(G / \beta, \colim^{E(G)} \A).\]
For $n=0$ we have
\[ H_0(G / \beta, \colim^{E(G)} \A) = \colim^{G / \beta} (\colim^{E(G)} \A) \cong \colim^{\frakL(G)} \A = H_0(G,\A) \]
by Proposition \ref{comp_of_colim}.  The transitivity of $\beta$ on $E(G)$ is sufficient to ensure that $\A \mapsto \colim^{E(G)} \A$ is exact, (see, for example, \cite[tag 04AX]{stacks}).  It follows that the sequence of functors $H_n(G / \beta, \colim^{E(G)} -)$ induces, from a short exact sequence  $\A \ra \B \ra \C$ of $G$--modules an exact sequence
\begin{align*} \dots  \ra H_{n+1}(G /\beta, & \colim^{E(G)} \C) \ra H_n(G / \beta ,  \colim^{E(G)} \A) \ra H_n(G / \beta,  \colim^{E(G)}\B) \\
& \ra H_n(G /\beta ,  \colim^{E(G)}\C)  \ra H_{n-1}(G /\beta,  \colim^{E(G)}  \A) \ra \dots \end{align*}
Now suppose that $\P$ is a projective $\frakL(G)$--module.  By Lemma
\ref{infl_preserves_epis} the expansion functor $\Mod{G / \beta} \ra \Mod{\frakL(G)}$ preserves epimorphisms, and so  its
left adjoint $\colim^{E(G)}$ preserves projectives.  Therefore $\colim^{E(G)} \P$ is projective, and for $n>0$
we have $H_n(G / \beta,\colim^{E(G)} \P)=0$.  
\end{proof}

\end{document}